\documentclass[12pt]{amsart}
\usepackage{amsmath}
\usepackage{geometry,amsfonts,amssymb,amsthm,txfonts,pxfonts,amscd,algorithmic,algorithm} 
\usepackage{refcheck}
\norefnames
\nocitenames
\def\struckint{\mathop{%
\def\mathpalette##1##2{\mathchoice{##1\displaystyle##2}%
  {##1\textstyle##2}{##1\scriptstyle##2}{##1\scriptscriptstyle##2}}%
\mathpalette
{\vbox\bgroup\baselineskip0pt\lineskiplimit-1000pt\lineskip-1000pt
\halign\bgroup\hfill$}
{##$\hfill\cr{\intop}\cr\diagup\cr\egroup\egroup}%
}\limits}
\newtheorem{theorem}{Theorem}[section]
\newtheorem{lemma}[theorem]{Lemma}
\newtheorem{corollary}[theorem]{Corollary}
\newtheorem{conjecture}[theorem]{Conjecture}
\newtheorem{definition}[theorem]{Definition}

\theoremstyle{remark}
\newtheorem{observation}[theorem]{Observation}
\newtheorem{remark}[theorem]{Remark}

\newtheorem{question}[theorem]{Question}

\newcommand{\integers}{\mathbb{Z}}

\newcommand{\reals}{\mathbb{R}}

\DeclareMathOperator{\tr}{tr}
\DeclareMathOperator{\acosh}{arccosh}
\DeclareMathOperator{\mcg}{Mod}

\DeclareMathOperator{\SL}{SL}
\DeclareMathOperator{\PSL}{PSL}

\DeclareMathOperator{\asin}{arcsin}
\DeclareMathOperator{\sech}{sech}
\DeclareMathOperator{\arctanh}{arctanh}

\DeclareMathOperator{\teich}{\mathcal{T}}

\begin{document}

\title{Geodesics with one self-intersection, and other stories}

\author{Igor Rivin}
\address{Department of Mathematics, Temple University, Philadelphia}
\address{School of Mathematics, Institute for Advanced Study, Princeton}
\email{rivin@temple.edu}
\thanks{The author would like to thank Henry Wilton, Alan Reid, Dennis Sullivan, and Andrew Odlyzko for enlightening discussions. A previous version of this paper appeared as \cite{preprintversion}}
\date{\today} 


\begin{abstract}
In this note we show that for any hyperbolic surface $S$, the number of geodesics  of length bounded above by $L$ in the mapping class group orbit of a fixed closed geodesic $\gamma$ with a single double point is asymptotic to $L^{\dim(\mbox{Teichmuller space of $S.$})}.$ Since closed geodesics with one double point fall into a finite number of $\mcg(S)$ orbits, we get the same asympotic estimate for the number of such geodesics of length bounded by $L,$ and systems of curves, where one curve has a self-intersection, or there are two curves intersecting once. We also use our (elementary) methods to do a more precise study of geodesics with a single double point on a punctured torus, including an extension of McShane's identity to such geodesics.

In the second part of the paper we study the question of when a covering of the boundary of an oriented  surface $S$ can be extended to a covering of the surface $S$ itself, we  obtain a complete answer to that question, and also to the question of when we can further require the extension to be a \emph{regular} covering of $S.$

We also analyze the question of the minimal index of a subgroup in a surface group which does not contain a given element. We show that we have a linear bound for the index of an arbitrary subgroup, a cubic bound for the index of a normal subgroup, but also poly-log bounds for each fixed level in the lower central series (using elementary arithmetic considerations) -- the results hold for free groups and fundamental groups of closed surfaces.
\end{abstract}
\maketitle
\tableofcontents

\section*{Introduction}

\subsection*{Notation} All surfaces considered in this paper are hyperbolic, orientable, of finite type (but not necessarily closed). We denote the \emph{mapping class group} of a surface $S$ by $\mcg(S).$ We denote the free group of rank $n$ by $F_n.$ We say that a function $f(x)$ is asymptotic to $g(x)$ at $a$ (usually $a=\infty$) if $\lim_{x\rightarrow a} f(x)/g(x) = 1. $ We say that $f(x)$ and $g(x)$ (usually defined on the positive reals or natural numbers) are of the same order of growth if there exist positive constants $a, b, c, d$ if $-a + b f(x) \leq g(x) \leq c + d f(x),$ for all $x$ in the domain of $f$ and $g.$ We shall denote the \emph{Teichm\"uller space} of a surface $S$ by $\teich(S).$ For a closed geodesic $\gamma,$ we denote the trace of the isometry corresponding to $\gamma$ by $t_\gamma = 2\cosh \ell \gamma/2.$

In this paper we consider a number of at first apparently unrelated questions. 

First, we study the set of closed geodesics with a single self-intersection (Section \ref{geodonegen}). We show that any such geodesic is contained in an embedded pair of pants (Theorem \ref{pantsthm}) and use that observation together with estimates on the length of the geodesic in terms of the lengths of the boundary components of the pair of pants (Theorem \ref{pantlength}) and results of M.~Mirzakhani (see \cite{mirzakhcurves,rivinmirzakh}) to show that the number of geodesics with a single self-intersection and length bounded by $L$ on a hyperbolic surface $S$ is \emph{asymptotic} to  $L$ raised to the dimension of the Teichmuller space of $S.$ We extend this result to the growth rate of \emph{one-multicurves} -- systems of pairwise-disjoint curves exactly one of which has a single self-intersection. As a side observation we find that on any hyperbolic surface any curve with one self-intersection has to have length at least $2\acosh 3$ (this is attained for a thrice punctured sphere) -- this result should be compared with related results of J. Hempel obtained in \cite{hemptrace}. Using a different technique (based on \cite{mcrivin1,mcrivin2}) we show that the number of elements shorter than $L$ in the $\mcg(T)$ orbit of a bouquet of two geodesic circles (where $T$ is a punctured torus) is asymptotic to $cL^2$ -- see Theorem \ref{asympeight}.

The celebrated \emph{McShane's identity}, proved by G.~McShane in his doctoral dissertation (see \cite{mcthesis,mcshane99}) states that 
\[
\boxed{
\sum\dfrac{1}{\exp(\ell(\gamma)) -1} = \frac12,
}
\]
where the sum is taken over all simple closed geodesics $\gamma$ with respect to an \emph{arbitrary} complete hyperbolic metric of finite area on a once-punctured torus. There has been a number of different proofs and extensions of McShane's result, due both to Greg McShane himself and others (see \cite{mcbowditch} for a nice short combinatorial proof, the ideas of which go back to Don Zagier's paper \cite{zagier}). McShane's identity was the crucial ingredient in Mirzakhani's ground-breaking work on computing the volume of moduli spaces of Riemann surfaces of finite type (with the Weil-Petersson metric) in \cite{mirzakhwp}. A number of extensions of McShane's identity currently exist, due to the work of a number of authors, including Greg McShane, Brian Bowditch, Maryam Mirzakhani, and Ser P. Tan and his collaborators. We extend all of these results to curves with a single self-intersection in Section \ref{mcshanesids}. As an example,
in the case of the punctured torus, we have the following extension:
\[
\boxed{
\sum\left(1-\sqrt{1-\left(\dfrac{6}{t_\gamma}\right)^2}\right) =2,
}
\]
where the sum is taken over all the geodesics with a single self-intersection (and does not depend on the hyperbolic structure).

One approach to understanding self-intersecting curves is to lift to a cover where the curve is simple (that this is always possible is the subject of G. Peter Scott's classic paper \cite{scottlerf}. We show that for geodesics with a single self-intersection a four-fold covering is always sufficient (see Section \ref{cover}), but the method of proof raises the question on when a given covering of a collection of curves on a surface can be extended to a covering of the whole surface. We study these questions in Sections \ref{coversec} and \ref{regcoversec}.

Finally, we are led to the following related question: It is well-known that free groups and fundamental groups of closed surfaces are residually finite. It is reasonable to ask, given an element $g$ of one of these groups $G,$ what the index of a subgroup (normal or otherwise) of $G$ not containing $g$ is. We are able to obtain a number of upper bounds, as follows:

For free and surface groups, given an element $g$ of length $n,$ there is a subgroup of index $O(n)$  not containing $g.$ This is Theorem \ref{bourabee}, which is proved by considering coverings.

For free groups, there is a normal subgroup of index exponential in $n$ which contains \emph{no} element of length $n.$ For surface groups, we are only able to bound the index by exponential of $O(n^2).$ This is the content of Theorems \ref{expthm} and \ref{expthmsurf}. Theorem \ref{expthm} is the result for free groups, and uses the results of Lubotzky-Phillips-Sarnak on expansion of Cayley graphs of special linear groups over finite fields. Theorem \ref{expthmsurf} concerns surface groups, and uses the result of Baumslag that surface groups are residually free (this can be made quantitative, as was pointed out to the author by Henry Wilton).

For free and surface groups there is a normal subgroup of index bounded by $O(n^3)$ which does not contain $g.$ This uses arithmetic representations of both free and surface groups and a little number theory. (Theorems \ref{nongamb}, \ref{nongambsurf}) 

If $g$ lies at the $k$-th level of the lower central series, for both free and surface groups we get a bound $O(\log^{k(k+1)/2} n)$ for the index of the normal subgroup not containing $g.$ This is the content of Theorems \ref{lowerng} and \ref{lowerngsurf}. The former is proved by constructing unipotent representations, the second follows from the former via the residual freeness approach. The idea of considering the lower central series comes from the (highly recommended) paper of J. Malestein and A. Putman \cite{maleput}

The above result may be put in context by Theorem \ref{avbourabee}, which states that (for free groups) the \emph{average} index of a subgroup not containing a given element is smaller than $3$ -- the proof gives the same result for normal subgroups, but the rank has to be at least four in the normal case. The argument uses the results of the author on distribution of homology classes in free groups, but the argument can be easily tweaked to work for surface groups.

\section{Geodesics with a single intersection on general hyperbolic surfaces}
\label{geodonegen}
Consider a hyperbolic surface $S$ and a geodesic $\gamma$ with exactly one double point $p.$ Let $\gamma_1$ and $\gamma_2$ be the two simple loops into which $\gamma$ is decomposed by $p,$ so that in $\pi_1(S, p)$ we can write $\gamma = \gamma_1 \gamma_2^{-1}.$ It is easy to see that 
$=\gamma_3 =(\gamma_1 \gamma_2)^{-1}$ is freely homotopic to a simple curve, and indeed:
\begin{theorem}
\label{pantsthm}
The geodesic $\gamma$ is contained in an embedded pair of pants whose boundary components are freely homotopic to $\gamma_1, \gamma_2, \gamma_3.$
\end{theorem}
\begin{proof}
Since $\gamma_1,\gamma_2, \gamma_3$ admit disjoint simple representatives, it is standard (see, for example, Freedman, Hass, Scott\cite{freedmanhassscott}) that the geodesic representative are disjoint, and obviously bound a pair of pants $\Pi. $Since a pair of pants is a convex surface, it follows that t $\gamma=\gamma_1 \gamma_2^{-1}$ is contained in it.
\end{proof}

Theorem \ref{pantsthm} establishes a bijective correspondence between closed geodesics with a single double point and embedded pairs of pants. 

We will eventually also need the following observation:
\begin{lemma}
\label{pantsinter}
Let $\gamma, \gamma_1, \gamma_2, \gamma_3$ be as in the statement of Theorem \ref{pantsthm}, and let $\delta$ be a simple geodesic disjoint from $\gamma.$ Then, $\delta$ is also disjoint from $\gamma_1, \gamma_2, \gamma_3,$ unless it \emph{is} one of those three curves.
\end{lemma}
\begin{proof}
Suppose the conclusion of the lemma does not hold. Then, the curve $\delta$ intersects $\Pi.$ Let $\delta_1$ be a connected component of $\delta \cap \Pi.$ If $\delta_1$ intersects some $\gamma_i$ and $\gamma_j$ for $i \neq j,$ then $\delta_1$ intersects $\gamma$ (since $\Pi \backslash \gamma$ has three connected components, each of them connecting exactly one of the $\gamma_i.$ If $\delta_1$ intersects only one of $\gamma_i,$ it is contained in one of the three connected components of $\Pi\backslash \gamma$ mentioned in the previous sentence, and thus, together with one of the segments of $\gamma_i$ that it intersects it bounds a geodesic bigon in $\mathbb{H}^2.$ Since there are no such bigons, the proof is complete.
\end{proof}

The next observation is:

\begin{theorem}
\label{pantlength}
Let  $P$ be a three-holed sphere with boundary components $a, b, c$. Let $\gamma$ be a closed geodesic in $P$ freely homotopic to $a^{-1}b.$ Then, the length of $\gamma$ satisfies:
\begin{equation}
\label{cosheq}
 \cosh \ell(\gamma)/2 = 2 \cosh \ell(a)/2 \cosh \ell(b)/2 + \cosh\ell(c)/2.
 \end{equation}
 \end{theorem}
 
 \begin{proof}
 By the Cayley-Hamilton Theorem,

\[
A^{-1} = \tr(A) I - A.
\]
So, 

\[
A^{-1} B = \tr(A) B - AB,
\]
and so
\begin{equation}
\label{trace1}
\tr(A^{-1} B) = \tr(A) \tr(B) - \tr(AB) = \tr(A) \tr(B) - \tr(C).
\end{equation}
The result follows from this, the relationship between trace and the translation length and the not-quite-obvious fact that we can take $A, B, C$ with traces negative (this follows from deep results of W. Goldman \cite{goldmanreps}, but in this case we need simply to have one case easy to compute -- the thrice-punctured sphere does nicely).
\end{proof}
\begin{corollary}
A closed geodesic with a single double point  on a hyperbolic surface $S$ is no shorter than $2\acosh 3 = 2\log(3 + 2 \sqrt{2}).\approx 3.52549$
Such a geodesic is exactly of length $2\acosh 3$ if and only if $S$ is a three-cusped sphere.
\end{corollary}
\begin{proof}
Since $\cosh(x)$ is monotonic for $x\geq 0,$ the result follows immediately from Eq. \eqref{cosheq}.
\end{proof}
We are only interested in the case where $\ell(a), \ell(b), \ell(c)$ are large, in which case we can write
\[
\ell(\gamma)/2 \approx \log(\exp((\ell(a)+\ell(b)/2) + \exp(\ell(c)/2).
\]
Now, the boundary of the pair of pants can be viewed as a multicurve on the surface $S$ of length $l(a, b, c) = \ell(a)+ \ell(b)+ \ell(c).$
Further, the three-component multicurves which bound pairs of pants fall into a finite number of mapping class group orbits. 

Letting $l(a, b) = \ell(a) + \ell(b),$ we see that 
\[
\ell(\gamma)/2 \approx \log(\exp(l(a, b)/2) + \exp(\ell(c)/2)).
\]
Let $l_1 = \max(l(a, b)/2), \ell(c)/2).$ and let $l(a, b, c)/2 = (k+1)l_1,$ where $k \leq 1.$
Then, 
\[
l(a, b)/2 < \log(\exp(l(a, b)/2) + \exp(ell(c)/2)) < l(a, b)/2 + \log(2).
\]
So, for large $l(a, b),$ we see that 
\[
\ell(\gamma)/l(a, b, c) \approx 1/(k+1).
\]
From this, and the results of Rivin( \cite{geomded})and Mirzakhani (\cite{mirzakhcurves}, see also \cite{rivinmirzakh})  we see that the order of growth of the number of geodesics with a single double point is the same as the order of growth of an orbit of a multicurve. By the results of Mirzakhani \cite{mirzakhcurves} on the \emph{equidistribution} of the mapping class orbit of a multicurve in measured lamination space (and hence the equidistribution of the various ratios of lengths), we get:
\begin{theorem}
\label{mainthm}
The number of the geodesics of length not exceeding $L$ in the mapping class group orbit of a geodesic with a single double point  on $S$ is asymptotic to 
$L^{\dim \teich(S)}.$
\end{theorem}
\begin{corollary}
\label{maincor}
The number of geodesics with a single double point of length not exceeding $L$ is asymptotic to 
$L^{\dim \teich(S)}.$
 \end{corollary}
 \begin{proof}
 The result is immediate from Theorem \ref{mainthm} and the observation that there is a finite number of mapping class orbits of geodesics with a bounded number of self-intersections.
 \end{proof}
 
 \subsection{Multicurves}
 Let us make the following definition:
 \begin{definition}
 A \emph{one-multicurve} on a surface $S$ is a collection of curves $\gamma_1, \dotsc, \gamma_k$ where
 \begin{enumerate}
 \item $\lambda_i \cap \lambda_j = \emptyset,$ for all $i, j.$
 \item $\lambda_i$ is a simple curve for all $i>1.$
 \item $\lambda_1$ has a single double point.
 \end{enumerate}
 A \emph{geodesic one-multicurve} on $S$ is a one-multicurve where the curves $\gamma_1, \dots, \gamma_k$ are geodesics. The \emph{length}  $\ell(\lambda)$ of a geodesic multicurve $\lambda$ is the sum of the lengths of the curves $\lambda_i.$
 \end{definition}
 With this definition in place, we have the following theorem:
 \begin{theorem}[growth of multicurves]
 \label{multigrowth}
Let $S$ be a hyperbolic surface, and  let $\lambda$ be a geodesic one-multicurve on $S.$ Consider the set $\Lambda(x) = \{\gamma \in \mcg(\lambda) \ \left| \ell(\gamma) \leq x \right.\}.$ Then
\[
|\Lambda(x)| \sim x^{\dim \teich(S)}.
\]
\end{theorem}
\begin{proof}
Let $\lambda = (\lambda_1, \dotsc, \lambda_k).$
Replace the self-intersecting curve $\lambda_1$ by the three curves whose existence is postulated in Theorem \ref{pantsthm} -- call these curves $\lambda_{k+1}, \lambda_{k+2}, \lambda_{k+3}.$
By Lemma \ref{pantsinter}, the object $\lambda^\prime = (\lambda_2, \dotsc, \lambda_k, \lambda_{k+1}, \lambda_{k+2}, \lambda_{k+3})$ is a multicurve (it is possible that some of the curves are covered multiple times), from which the one-multicurve $\lambda$ can be recovered in the obvious way, and the result follows by the argument used to prove Theorem \ref{mainthm}.
\end{proof}
\begin{corollary}
\label{multicor}
The number of one-multicurves  on the surface $S$ of length not exceeding $x$ is asymptotic to $x^{\dim \teich(S)}.$
\end{corollary}
\begin{proof}
There is a finite number of mapping class orbits of one-multicurves. The result now follows by applying Theorem \ref{multigrowth} to each of them.
\end{proof}
 
\section{Punctured tori}
\label{ptorus}
The results of the previous section have particularly a particularly nice form when the surface $S$ is a once-punctured torus. A punctured torus $T,$ can be cut along a simple closed geodesic $\gamma$ into a sphere with two boundary components (each of length $\ell(\gamma)$) and one cusp. This means that to each simple geodesic $\gamma$ we can associate two geodesics $\gamma_1$ and $\gamma_2$ both of which have a single self intersection, are of the same length,  and if the translation corresponding to $\gamma$ is $A$ and that corresponding to $\gamma_1$ is $B,$ then, from Eq.\eqref{trace1},
\begin{equation}
\label{trace2}
\tr(B) = 3 \tr(A).
\end{equation}
If $\ell(\gamma)$ is large, Eq. \eqref{trace2} implies that $\ell(\gamma_1) \approx \ell(\gamma) + 2\log 3.$ Set $N_0(L, T)$ be the number of simple geodesics of length bounded above by $L$ on the punctured torus $T,$ and let $N_1(L, T)$ be the number of geodesics with a single double point on the same torus, then.
\begin{theorem}
\label{simpnosimp}
\[N_0(L, T) \sim N_1(L, T)/2.\]
\end{theorem}
\begin{remark}
 The same result holds when instead of a punctured torus we consider a torus with a geodesic boundary component.
\end{remark}
In addition, the two geodesics $\gamma_1$ and $\gamma_2$ are \emph{homologous} to $\gamma,$ which gives us the following result:
\begin{theorem}
\label{torushomology}
There are exactly two simple closed geodesics with one self intersection in each primitive homology class on the punctured torus.
\end{theorem}
\begin{remark}
Combinatorial results on geodesics with a single self-intersection on a punctured torus have been obtained by D. Crisp and W. Moran in \cite{crisp1}.
\end{remark}
\subsection{Pairs of curves with one intersection}
Consider a punctured torus $T$ and a pair of simple closed geodesics $\gamma_1,\gamma_2$ which intersect exactly once. Let $\gamma = \gamma_1 \cup \gamma_2,$ and define  $\ell(\gamma) = \ell(\gamma_1) + \ell(\gamma_2).$ We first note the following observation:
\begin{theorem}
\label{singleorb}
All pairs of curves $\gamma$ as defined above lie in a single mapping class orbit.
\end{theorem}
\begin{proof}
There is exactly one simple closed geodesic in each integral homology class on the punctured torus, and the intersection number of two curves whose classes are $(a, b)$ and $(c, d)$ is given by $\bigl| \begin{smallmatrix}
a&b\\
c&d \end{smallmatrix} \bigr|.$ The previous statement is well-known, see, for example, \cite{mcrivin1}. Thus, pairs of curves $\gamma$ correspond exactly to images of pairs of standard generators of the fundamental group of the punctured torus under the mapping class group of the standard torus (which is well-known to be isomorphic to $\SL(2, \mathbb{Z}),$ where the isomorphism maps a mapping class to the matrix $\bigl(\begin{smallmatrix} a & b \\ c & d\end{smallmatrix}\bigr)$ as above.
\end{proof}
We now recall the results of \cite{mcrivin1,mcrivin2}, where it is shown, in particular, that for a fixed hyperbolic structure on the punctured torus $T,$ the length of the (unique) simple closed curve in a homology class $(a, b)$ is given by $\beta(a, b),$ where $\beta$ is the restriction to $\integers^2$ of a Banach space norm on $\reals^2,$ and so $\ell(\gamma) = \beta(a, b) + \beta(c, d).$ We can now use the results of \cite{eskinmcmullen} to show the following result:
\begin{theorem}
\label{asympeight}
The number $N_x$ of pairs of curves $\gamma$ with $\ell(\gamma) < x$ is asymptotic to $x^2.$
\end{theorem}
\begin{proof}
By the discussion above, this reduces to counting matrices $M = \bigl(\begin{smallmatrix} a & b \\ c & d\end{smallmatrix}\bigr) \in \SL(2, \integers),$ such that $\beta(M)= \beta(a, b) + \beta(c, d) \leq x.$ 
Since, by \cite{eskinmcmullen}[Proposition 6.2] the sets $B_x = \{M | \beta(M) \leq x$ are well-rounded, the asymptotic growth rate of $\SL(2, \integers) \cap B_x$ is the same as that of $\SL(2, \integers) \cap N_x,$ where $N_x$ is the set of matrices of Frobenius norm bounded by $x.$ The latter set is known to grow quadratically, by the results of M. Newman \cite{newmanmats} (these results, of course, are a very simple special case of the results of \cite{eskinmcmullen,drs93}, but it is a little tricky to extract it from these last two papers).
\end{proof}
The statement ofTheorem \ref{asympeight} has no error bounds (nor an explicit proportionality constant, but that will depend on the hyperbolic metric on the punctured torus)The author believes that the following conjecture holds:
\begin{conjecture}
\label{asympconj}
The error term in the asymptotic approximation for $N_x$ is at least of the order of $O(x),$ and at most of the order of $O(x \log x).$
\end{conjecture}
Some evidence for Conjecture \ref{asympconj} is given by the results discussed in \cite{mcrivin2}.

\section{McShane's identities}
\label{mcshanesids}
In his doctoral dissertation \cite{mcthesis}, Greg McShane obtained the following amazing identity (which almost immediately came to be known as ``McShane's identity":
\begin{theorem}[McShane's Identity]
On a punctured torus with a complete hyperbolic metric of finite area we have
\begin{equation}
\label{mcidentity1}
\sum \dfrac{1}{e^{\ell(\gamma)} + 1} = \dfrac12,
\end{equation}
where the sum is taken over all the simple closed geodesics.
\end{theorem}

McShane's Identity \eqref{mcidentity1} can be rewritten by using the trace $t_\gamma$ of the hyperbolic translation corresponding to $\gamma$ as follows:
\begin{equation}
\label{mcidentity1trace}
\sum \left( 1 - \sqrt{1-\left(\dfrac{2}{t_\gamma}\right)^2}\right) = 1,
\end{equation}
using the well-known fact (see, for example, \cite{beardongroups}) that 
\begin{equation}
\label{basictrace}
\tr(\gamma) = 2 \cosh \frac{\ell(\gamma)}2
\end{equation}

Using Eq. \eqref{mcidentity1trace}  and Eq.~\eqref{trace2}, we obtain the ``McShane's identity for curves with one self-intersection''
\begin{theorem}
\label{mcidentity1int}
\begin{equation}
\label{mc2}
\sum\left(1-\sqrt{1-\left(\dfrac{6}{t_\gamma}\right)^2}\right) =2,
\end{equation}
where the sum is taken over all the simple closed geodesics on the punctured torus with one double point.
\end{theorem}

The story, however, only just begins (we will not be telling it in order). Greg McShane then proceeded to refine his result, as follows (see \cite{mcweiertorus}): The simple geodesics on a punctured torus are uniquely determined by their integral homology class (which is always primitive), and so fall into three families corresponding to the three possible $\mod 2$ homology classes $c_1 = (0, 1), c_2 = (1, 0), c_3 = (1, 1)$ -- the classes correspond to the fixed points of the elliptic involution on the punctured torus (Weierstrass points), but we will not be using this. McShane's result then is:
\begin{theorem}[McShane's Weierstrass point identity]
\label{mcweier}
\begin{equation}
\label{mcidentityw}
\sum_{\gamma \in c_i} \asin (\sech \frac{\ell(\gamma)}2) = \dfrac{\pi}{2},
\end{equation}
where the sum is taken over all simple closed geodesics int he class $c_i$ (where $i=1, 2, 3$).
\end{theorem}

we will restate Identity \eqref{mcidentityw} as:
\begin{equation}
\label{mcidentitywtr}
\sum_{\gamma \in c_i} \asin \frac{2}{t_\gamma} = \dfrac{\pi}{2}.
\end{equation}
Now, Eq. \ref{trace2}, Identity \ref{mcidentitywtr}, and Theorem \ref{torushomology} immediately imply:
\begin{theorem}
\label{mcidentityw1int}
On a torus with one puncture equipped with a complete hyperbolic metric of finite area, the geodesics with one self-intersection fall into the three $\mod 2$ homology classes $c_1, c_2, c_3,$ and furthermore,
\begin{equation}
\label{weier1int}
\sum_{\gamma \in c_i} \asin \frac6{t_\gamma} = \pi,
\end{equation}
where the sum is taken over all the geodesics with one self-intersection lying in the homology class $c_i$ (where $i=1, 2, 3$).
Furthermore, 
\begin{equation}
\label{weiera1intall}
\sum_\gamma \asin \frac6{t_\gamma} = 3 \pi,
\end{equation}
where the sum is taken over \emph{all} geodesics with one self-intersection.
\end{theorem}

What happens if the surface is still a torus, but instead of a cusp it now possesses a boundary component of length $\delta$ or  \emph{cone angle} $\theta.$ This case was also considered by G. McShane, but, to my knowledge, first appeared in the paper \cite{tanjdg}. We will summarize all the results in the two theorems below:
\begin{theorem}[McShane's "classic" identity generalized]
\label{classicgen}
Let $T$ be a perforated torus equipped with a hyperbolic metric with one geodesic boundary component $\delta.$ Then
\begin{equation}
\label{macgenbdry}
\sum_\gamma 2 \arctanh 
\left( \frac {\sinh \frac{\ell(\delta)}2}{\cosh \frac{\ell(\delta)}2 + \exp \ell(\gamma)}\right) = \dfrac{\ell(\delta)}2,
\end{equation}
where the sum is taken over all simple closed geodesics on $T.$
If, instead, $T$ is a torus equipped with a hyperbolic metric with one cone point, with cone angle equal to $\theta,$ then 
\begin{equation}
\label{macgencone}
\sum_\gamma 2 \arctan \left( \frac {\sin \frac\theta2}{\cos \frac\theta2 + \exp \ell(\gamma)}\right) = \dfrac\theta2,
\end{equation}
Where the sum is, again, taken over all simple closed geodesics on $T.$
\end{theorem}
Notice that the classic McShane's identity is recovered from the first statement as the length goes to $0,$ or from the second as the angle goes to $0.$
\begin{theorem}
\label{weiergen}
Let $T$ be a perforated torus equipped with a hyperbolic metric with one geodesic boundary component $\delta.$ Then
\begin{equation}
\label{weiergenbdry}
\sum_{\gamma \in c_i} \arctan \frac{\cosh \frac {\ell(\delta) } 4}{\sinh \frac{\ell(\gamma)} 2} = \frac \pi 2,
\end{equation}
where the sum is taken over all the geodesics lying in a fixed \emph{nontrivial} $\mod 2$ homology class.

If $T$ is a torus equipped with a hyperbolic metric with one cone point of angle $\theta,$ then
\begin{equation}
\label{weiergencone}
\sum_{\gamma \in c_i} \arctan \frac{\cos \frac \theta 4}{\sinh \frac{\ell(\gamma)}2} = \frac \pi 2,
\end{equation}
with the sum, again, taken over all geodesics lying in a fixed nontrivial $\mod 2$ homology class.
\end{theorem}
In order to adapt Theorems \ref{classicgen} and \ref{weiergen} to our situation of curves with one self intersection, we recall Formulas \eqref{trace1} and \eqref{trace2}. The latter, in the more general case considered here will have the form
\begin{equation}
\label{trace2nocusp}
\tr B = (\tr C +1) \tr A,
\end{equation}
where $C$ is the isometry corresponding to the kind of boundary we are having. If $C$ is a geodesic boundary component $\delta,$ then $\tr C = 2 \cosh \ell(\delta)/2,$ while if $C$ is a cone point of angle $\theta,$  then $\tr C = 2 \cos \theta/2 + 1.$ As before, to each simple closed curve $\gamma$ there correspond \emph{two} curves with a single self intersection (both of the same length $\ell,$ satisfying $\tr B = 2 \cosh \ell/2.$)

In addition, we have, from the usual trigonometric formulas:
\begin{gather*}
\cos \frac \theta 4 =
 \sqrt{\frac{1+\cos \theta}2}= 
\sqrt{\frac{1+ \frac12 \tr C}2}=\frac12 \sqrt{2+\tr C}
\\
\cosh \frac l 4 =
 \sqrt{\frac{1+\cosh l}2}= 
\sqrt{\frac{1+ \frac12 \tr C}2}=\frac12 \sqrt{2+\tr C}
\\
\sin \theta/2 = \sqrt{1-\cos^2\theta/2} = \frac12 \sqrt{4-\tr^2C}
\\
\sinh l/2 = \sqrt{\cosh^2 l/2 - 1} = \frac12 \sqrt{\tr^2 C - 4}
\\
\exp l = \left(\cosh l/2 + \sinh l/2\right)^2 = 
 \frac12 \tr^2 C - 1 + \frac 12 \tr C 
 \sqrt{\tr^2 C - 4}
\end{gather*}
Using this, we can rewrite Theorems \ref{classicgen} and \ref{weiergen} as follows:
\begin{theorem}[McShane's "classic" identity generalized (one self-intersection)]
\label{classicgen1int}
Let $T$ be a perforated torus equipped with a hyperbolic metric with one geodesic boundary component $\delta.$ Then
\begin{equation}
\label{macgenbdry1int}
\sum_\gamma \arctanh \left( \frac{(t_\delta + 1)\sqrt{t^2_\delta - 4}}{(t_\delta -2)(t_\delta + 1) + t_\gamma + t_\gamma \sqrt{t^2_\gamma - 4(t_\delta +1)^2} } \right)= \dfrac{\ell(\delta)}2,
\end{equation}
where the sum is taken over all simple closed geodesics on $T.$ 
If, instead, $T$ is a torus equipped with a hyperbolic metric with one cone point, with cone angle equal to $\theta,$ then 
\begin{equation}
\label{macgencone1int}
\sum_\gamma \arctan \left( \frac{(t_ \delta + 1)\sqrt{-t^2 \delta +4}}{(t \delta -2)(t_\delta + 1) + t^2_\gamma + t_\gamma \sqrt{t^2_\gamma - 4(t_\delta +1)^2} } \right)= \dfrac{\theta}2,
\end{equation}
Where the sum is, again, taken over all simple closed geodesics on $T,$ the symbol $\Delta$ denotes the (elliptic) isometry corresponding to the cone point.\end{theorem}

\begin{theorem}[McShane's Weierstrass point identity generalized for curves with one self-intersection]
\label{weiergen1int}
Let $T$ be a perforated torus equipped with a hyperbolic metric with one geodesic boundary component $\delta.$ Let $\Delta$ be the corresponding matrix in $\SL(2, \mathbb R).$ Then
\begin{equation}
\label{weiergenbdry1int}
\sum_{\gamma \in c_i} \arctan \frac{(t_\delta +1) \sqrt{2 + t_\delta}}{\sqrt{t_\gamma - 4 (t_\delta + 1)^2}} = \pi
\end{equation}
where the sum is taken over all the geodesics lying in a fixed \emph{nontrivial} $\mod 2$ homology class.

If $T$ is a torus equipped with a hyperbolic metric with one cone point of angle $\theta,$ and $\Theta$ is the corresponding (elliptic) element in $\SL(2, \mathbb R),$ then
\begin{equation}
\label{weiergencone1int}
\sum_{\gamma \in c_i} \arctan \frac{(t_\delta +1) \sqrt{2 + t_\delta}}{\sqrt{-t^2_\gamma +4 (t_\delta + 1)^2}} = \pi
\end{equation}
with the sum, again, taken over all geodesics lying in a fixed nontrivial $\mod 2$ homology class.
\end{theorem}

\subsection{Punctured torus bundles}
In his paper \cite{bhbundles} Brian Bowditch proved the following result.
\begin{theorem}
\label{bowditch1}
Let $M$ be a hyperbolic $3$-manifold which fibers over the circle $S^1$ with fiber a once-punctured torus. Then, 
\[
\sum_{\gamma \in S} \dfrac1{1+\exp(\ell(\gamma))} = 0,
\]
where the set $S$ is the set of all geodesics in $M$ which correspond to simple closed curves in the fiber, and $\ell(\gamma)$ is the \emph{complex length} of $\gamma.$
\end{theorem}
This translates without change to the following theorem:
\begin{theorem}
\label{bowditch2}
Let $M$ be a hyperbolic $3$-manifold which fibers over the circle $S^1$ with fiber a once-punctured torus. Then, 
\[
\sum_{\gamma \in S} \left(1 - \sqrt{1 - \left(\frac6{t_\gamma}\right)}\right)= 0,
\]
where the set $S$ is the set of all geodesics in $M$ which correspond to closed curves with one self-intersection  in the fiber, and $t_\gamma$ is the trace of the hyperbolic isometry corresponding to $\gamma.$
\end{theorem}

\subsection{Surfaces of higher genus}
In his paper \cite{mcshane99} Greg McShane proved an extension of his identity for the punctured torus to surfaces with punctures. To state the most general form, we need a few definitions (our exposition is borrowed from M. Mirzakhani's Thesis \cite{mirzathesis}).

Let $S$ now be a surface of genus $g$ with $n$ punctures.  Let $\mathcal{F_i}, \quad 1\leq i \leq n,$ be the set of \emph{unordered} pairs of isotopy classes of curves $(\alpha_1, \alpha_2)$ which, together with the puncture $p_i$ bound a pair of pants. Further, let $\mathcal{F}_{i, j}$ be the set of simple closed curves which, together with the punctures $p_i, p_j$ bound a pair of pants. Then, the following identity holds:
\begin{equation}
\label{mchighergen}
\sum_{(\alpha_1, \alpha_2) \in \mathcal{F}_1}
\dfrac1{1+\exp(\frac12(\ell(\alpha_1) + \ell(\alpha_2)))}
+ \sum_{j=2}^n 
\sum_{\gamma \in \mathcal{F}_{1,j}}
\dfrac1{1+ \exp(\frac12 \ell(\gamma))} = 1.
\end{equation}

To extend Identity \eqref{mchighergen}, we first write it in terms of the traces of the corresponding elements, thus:
\begin{equation}
\label{mchighergentr}
\sum_{(\alpha_1, \alpha_2) \in \mathcal{F}_1}
\dfrac1{1+\left(t_{\alpha_1}+\sqrt{t_{\alpha_1}^2-1}\right)\left(t_{\alpha_2}+\sqrt{t_{\alpha_2}^2-1}\right)}
+ \sum_{j=2}^n 
\sum_{\gamma \in \mathcal{F}_{1,j}}
\dfrac1{1+t_\gamma + \sqrt{t_\gamma^2 - 1}} = 1
\end{equation}
Now we note that on a pair of pants with one cusp and cuffs $\alpha_1$ and $\alpha_2,$ there are two geodesics (call them $\gamma_1$ and $\gamma_2$) with one self-intersection winding once around the cusp. The traces of the corresponding matrices satisfy:
\begin{gather*}
t_{\gamma_1} = 2 t_{\alpha_1} + t_{\alpha_2},\\
t_{\gamma_2} = 2 t_{\alpha_2} + t_{\alpha_1}.
\end{gather*}
It follows that
\begin{gather*}
t_{\alpha_1} =\frac13(2t_{\gamma_1} - t_{\gamma_2}),\\
t_{\alpha_2} = \frac13(2t_{\gamma_2}-t_{\gamma_1}).
\end{gather*}
On a pair of pants with \emph{two} cusps and one cuff $\gamma,$ if $\delta$ is the geodesic with one self-intersection surrounding both the cusps,
then
\[
t_\gamma=t_\delta - 4.
\]
These relationships lead to the following identity for once punctured curves (remember that the curves in the first summation are pairs of geodesics with one self-intersection surrounding the cusp, while the curves in the second summation are geodesics with one self-intersection surrounding \emph{both} cusps.

\begin{multline}
\label{mchighergentr1int}
\sum_{(\gamma_1, \gamma_2) \in \mathcal{F}_1}
\dfrac9{9+\left(2t_{\gamma_1} - t_{\gamma_2}+\sqrt{(2t_{\gamma_1} - t_{\gamma_2})^2-9}\right)\left(2 t_{\gamma_2} - t_{\gamma_1}+\sqrt{(2t_{\gamma_2}-t_{\gamma_1})^2-9}\right)}\\
+ \sum_{j=2}^n 
\sum_{\delta \in \mathcal{F}_{1,j}}
\dfrac1{t_\gamma -3 + \sqrt{(t_\delta-4)^2 - 1}} = 1
\end{multline}
M. Mirzakhani (\cite{mirzathesis}) generalizes identity \eqref{mchighergen} to surfaces with boundary components instead of cusps, and the identity \eqref{mchighergentr1int} can be easily generalized to that setting -- we leave the computation as an exercise for the reader.
\section{Removing intersections by covering}
\label{cover}
It is a celebrated result of G. Peter Scott \cite{scottlerf} that for any hyperbolic surface $S$ and any closed geodesic $\gamma,$ there is a finite cover $\pi \tilde{S} \rightarrow S,$ such that the lift $\tilde{\gamma}$ to $\tilde{S}$ is simple (non-self-intersecting).  Since for any $k,$ there is a finite number of mapping classes of geodesics on $S$ with no more than $k$ self-intersections, and the minimal degree of $\pi$ corresponding to a curve $\gamma$ is invariant under the mapping class group, it follows that there exists some bound $d_S(k)$ so that one can ``desingularize''  any curve with up to $k$ self-intersections by going to a cover of degree at most $k.$ Unfortunately, Scott's argument appears to give no such bound. 

\begin{remark} Ilya Kapovich has told the author that a similar result follows from J.~Stallings' theory of subgroup graphs \cite{stallingsgraphs}.
\end{remark}


The question of providing good bounds for $d_S(k)$ is, as far as I can say, wide open.  Here we will attempt to start the ball rolling by giving sharp bounds for $d_(1).$

Our first observation is that if $S$ is a three-holed sphere, then $d_S(1) = 2.$ To prove this we consider (without loss of generality) the case where $S$ is a thrice cusped sphere (the quotient of the hyperbolic plane by $\Gamma(2).$). The proof then is contained in the diagram below. As can easily be seen, the lifts of two of the boundary components (say, $A$ and $B$) are connected, and the lift of $C$ has two connected components. 

Now, suppose that $\gamma$ is contained in a closed surface $T.$ Since $\gamma$ has a three-holed sphere neighborhood $S,$  if the covering map described above extends to all of $T,$ then $d_T(1) = 2.$ However, there are obviously examples where the map does \emph{not} extend (f when $T\backslash S$ has three connected components, or, more generally, when the boundary one of the components of $T\backslash S$ has one connected component, the lift of which is connected -- since the Euler characteristic of a surface with one boundary component is odd, such a surface is not a double cover. It is easy to see that in all other cases the double cover does extend). It is, however, clear, that a further double cover removes the obstruction, and we obtain the following result:
\begin{theorem}
For any  oriented hyperbolic  surface $S$ and a geodesic $\gamma\subset S$ with a single double point, there is a four-fold cover of $S$ where $\gamma$ lifts to a simple curve.
\end{theorem}

\begin{remark}
The result is not sharp for some surfaces with boundary (for example, the thrice punctured sphere). The result is vacuously true for the 2-sphere and the 2-torus equipped with metrics of constant curvature, since \emph{all} the geodesics for those metrics are simple.
\end{remark}

\section{Extending covering spaces}
\label{coversec}
The results of the previous section suggest the following question:
\begin{question}
\label{coverq}
Given an oriented surface $S$ with boundary $\partial S,$ and a covering map of $1$-manifolds 
$\pi: \widetilde{C} \rightarrow \partial S$ the fibers of which have constant cardinality $n.$ When does $\pi$ extend to a covering map $\Pi: \widetilde{S}\rightarrow S,$ where $\partial \widetilde{S} \simeq \widetilde{C}?.$
\end{question}

It turns out that Question \ref{coverq} has a complete answer --  Theorem \ref{sullthm} below ( due, essentially to D. Husemoller \cite{husemollercovers}).  First, we note that to every degree $n$ covering map $\sigma: X\rightarrow Y$ we can associate a permutation representation $\Sigma: \pi_1(Y) \rightarrow S_n.$ Further, two coverings $\sigma_1$ and $\sigma_2$ are equivalent if and only if the associated representations $\Sigma_1$ and $\Sigma_2$ are conjugate (see, for example, \cite{hatcheralgtop}[Chapter 1] for the details). This means that the boundary covering map $\pi$ is represented by  a collection of $k=|\pi_0(\partial S)|$ conjugacy classes $\Gamma_1, \dotsc, \Gamma_k$  in $S_n.,$ each of which is the conjugacy class of the image of the generator of the fundamental group of the corresponding component under the associated permutation representation.
\begin{theorem}
\label{sullthm}
A covering $\pi:\widetilde{C}\rightarrow C=\partial S$ extends to a covering of the surface $S$ if and only if the following conditions hold:
\begin{enumerate}
\item $S$ is a planar surface (that is, the genus of $S$ is zero and there exists a collection $\{\sigma_i\}_{i=1}^k$ of elements of the symmetric group $S_n$ with $\sigma_i \in \Gamma_i$ such that 
$\sigma_1 \sigma_2 \dots \sigma_k = e,$ where $e\in S_n$ is the identity.
\item $S$ is not a planar surface, and the sum of the parities of $\Gamma_1, \dotsc, \Gamma_k$ vanishes.
\end{enumerate}
\end{theorem}
\begin{proof}
In the planar case, the fundamental group of $S$ is freely generated by the generators $\gamma_, \dotsc, \gamma_{k-1}$  fundamental groups of (any) $k-1$ of the boundary components.  All $k$ boundary components satisfy $\gamma_1 \dotsc \gamma_k = e,$ whence the result in this case.

In the nonplanar case, let us first consider the case where $k=1.$ The generator $\gamma$ of the single boundary component is then a product of $g$ commutators (where $g$ is the genus of the surface, and so $\Sigma(\gamma)$ is in the commutator subgroup of $S_n,$ which is the alternating group $A_n,$ so the class $\Sigma(\gamma)$ has to be even. On the other hand, it is a result of O. Ore \cite{orecomm} that any even permutation is a commutator, $\alpha \beta \alpha^{-1} \beta^{-1}$ and thus sending some pair of handle generators to $\alpha$ and $\beta$ respectively and the other generators of $\pi_1(S)$ to $e$ defines the requisite homomorphism of $\pi_1(S)$ to $S_n.$

If $k>1,$ the surface $S$ is a connected sum of a surface of genus $g>0$ (by assumption) and a planar surface with $k$ boundary components.  let $\gamma$ be the ``connected summing'' circle. By the planar case, there is no obstruction to defining $\Pi$ on the planar side (since $\gamma$ is not part of the original data). However, $\Sigma(\gamma)$ will be the inverse of the product of elements $\sigma_i \in \Gamma_i$ and so its parity will be the sum of the parities of $\Gamma_i.$ To extend the cover to the non-planar side of the connected sum, it is necessary and sufficient for this sum to be even.
\end{proof}

Some remarks are in order. The first one concerns the planar case of Theorem \ref{sullthm}. It is not immediately obvious how one might be able to figure out whether given some conjugacy classes in the symmetric group, there are representatives of these classes which multiply out to the identity. Luckily, there is the following result of Frobenius (see \cite{serregalois}[p. 69])
\begin{theorem}
\label{frob}
Let $C_1, \dotsc, C_k$ be conjugacy classes in a finite group $G.$ The number $n$ of solutions to the equation $g_1 g_2 \dots g_k = e,$ where $g_i\in C_i$ is given by 
\[
n = \dfrac{1}{|G|} |C_1| \dots |C_k| \sum_{\chi} \dfrac{\chi(x_1) \dots \chi(x_k)}{\chi(1)^{k-2}},
\]
where $x_k \in C_k$ and the sum is over all the complex irreducible characters of $G.$
\end{theorem}
Special cases of the planar case are considered in \cite{Kulkarnicovers}; enumeration questions for covers are considered in a number of papers by A.~Mednykh -- see \cite{Mednykhenum} and references therein.

The second remark is on Ore's result that every element of $A_n$ is even. This result was strengthened by E. Bertram in \cite{bertramcycles} and, independently and much later, by H. Cejtin and the author in \cite{henry} (the second argument has the virtue of being completely algorithmic, the first, aside from being 30 years earlier, proves a stronger result) to the statement that every even permutation $\sigma$ is the product of two $n$-cycles (Bertram actually shows that it is the product of two $l$ cycles for any 
$l\geq (M(\sigma) + C(\sigma))/2,$ where $M(\sigma)$ is the number of elements moved by $\sigma$ while $C(\sigma)$ is the number of cycles in the cycle decomposition of $\sigma$.

The significance of this to coverings is that we have a very simple way of constructing a covering of a surface with one boundary component with specified cycle structure of the covering of the component, as follows.

First, the proof of Theorem \ref{sullthm} shows that the construction reduces to the case where $g=1,$ so that we are constructing a covering of a torus with a single perforation.

Suppose now that the permutation can be written as $\sigma\tau\sigma^{-1}\tau^{-1},$ where $\sigma$ is an $n$-cycle. This means that the "standard" generators of the punctured torus group go to $\sigma$ and $\tau,$ respectively. To construct the cover, then, take the standard square fundamental domain $D$ for the torus (the puncture is at the vertices of the square), then arrange $n$ of these fundamental domains in a row, and then a strip, by gluing the rightmost edge to the leftmost edge. Then, for each $i,$ the upper edge of the $i$-th domain from the left ($D_i$) is glued to the lower edge of $D_{\tau(i)}.$

In the joint paper \cite{drosterivin} with Manfred Droste we extend the results of this section to \emph{infinite} covers.

\section{Quantifying residual finiteness} 
\label{resfin}
Khalid Bou-Rabee in \cite{khalid2} has analyzed the following question: given a residually finite group $G$ and an element $g\in G,$ how high an index subgroup $H< G$ must one take so that $g \notin H,$ in terms of the word-length of $g.$ Bou-Rabee answers the questions for important classes of groups, including arithmetic lattices and nilpotent groups.
Here we wish to point out that for surface groups (including free groups) we have the following bound:
\begin{theorem}
\label{bourabee}
Given an element $g\in G$ of word length $l(g),$ there is a subgroup $H$ of $G$ of index of order $O(l(g)),$ such that $g \notin H.$
\end{theorem}
\begin{proof}
Let $F$ be a surface such that $\pi_1(F) = G.$
It is obviously equivalent to construct a cover  $\widetilde{F}$ of the surface $F$ whose fundamental group is $H,$ such that $g$ does not lift to $\widetilde{F}.$ There are two cases. The first is when the geodesic $\gamma(g)$ in the conjugacy class of $g$ is simple. In that case there are two further cases:  the first arises when $\gamma(g)$ is homologically nontrivial. In that case, there is a geodesic $\beta$ transversely intersecting $\gamma$ in one point. Cutting $F$ along $\beta$ and then doubling gives us a double cover where $g$ does not lift. The second case is when $\gamma(g)$ bounds. In that case, cut along $\gamma(g)$ to obtain two surfaces with boundary. each of them admits a (connected) cover which restricts to a triple connected cover over $\gamma(g).$ Gluing along this cover, we obtain a cover of $F$ where $g$ does not lift.

The second case is when $g$ is \emph{not} simple. In this case, an examination of G.~P.~Scott's argument in \cite{scottlerf} shows that there is a cover $\widetilde{F}$ of $F$ of index linear in the word length of $g$ where the lift of $g$ is a simple. The first case analyzed above then  completes the argument.
\end{proof}

The usual definition of residual finiteness is the following: a group $G$ is residually finite, if for every $g\in G$ there is a homomorphism $\psi_g: G \rightarrow H,$ where $H$ is finite and such that $\psi_g(g) \neq e.$ In other words, it postulates the existence of a \emph{normal} subgroup of finite index ($\ker \psi_g$) which does not contain $g.$ Now, since every subgroup of index $k$ in an infinite group $G$ contains a normal subgroup of index $k!$ (index in $G,$ that is) the two points of view on residual finiteness are logically equivalent \emph{if} we don't care too much about the index. If we do, note that Theorem \ref{bourabee} gives us the following Corollary:
\begin{corollary}
\label{bouracor}
Let $G$ be a surface group. Given an element $g \in G$ of word length $l(g),$ there is a \emph{normal} subgroup $H$ of index at most $(c l(g))!$ which does not contain $g.$
\end{corollary}

Corollary \ref{bouracor} can be improved considerably for free groups:
\begin{theorem}
\label{expthm}
Consider the free group on $k$ letters $F_k$, and let $n  > 1.$  There exists a normal subgroup $H_n$ of $F_k$ of index $f(n)$ which contains \emph{no} non-trivial elements of word length smaller than $n,$  where the index $f(n)$ can be bounded by
\[
f(n) \leq c (2k-1)^{3n/4}.
\]
for some constant $c.$
\end{theorem}
\begin{proof}
We first note that if we have a homomorphism $\phi$ of $F_k=\langle a_1, \dots, a_k\rangle$ onto a finite group $H$ with the Cayley graph $C_H$ of $H$ with respect to the generating set $\phi(a_1), \dotsc, \phi(a_k),$ then no word in $F_k$ shorter than the girth of $C_H$ is in the kernel of $\phi.$

We now use the following result of Lubotzky, Phillips, and Sarnak (\cite{lps}, see \cite{dsv} for an expository account): 
\begin{citation}
For $p, q$ prime, with $p\geq 5$ and $q \gg p$ and $p$ a quadratic non-residue mod $q$ there is a symmetric generating set $S$ of the  group $\PSL(2, q)$ of order $p+1$ such that the Cayley graph of $\PSL(2, q)$ has girth no smaller than $4\log_p q - \log_p 4.$
\end{citation}

 It follows no element of $F_{(p+1)/2}$ of length shorter than $n(p, q)=4\log_p q - \log_p 4$ is killed by the homomorphism $\phi_q$ that sends the free generators of $F_{(p+1)/2}$ and their inverses to $S.$
Since the order of $\PSL(2, q)$ has order $m_qq(q^2-1)/2 \sim q^3/2,$ which we can write down as 
\[
m_q=(4p^{n(p, q)})^{3/4}=2^{3/2} n(p, q).
\]
If $k\neq{(p+1)/2}$ for some prime $p,$ we can find a subgroup of small index in $F_k$ which \emph{is} a free group on $(p+1)/2$ letters for some prime $p.$ Using Dirichlet's theorem on primes in arithmetic progressions we can then find a suitable $q.$ 
\end{proof}

Theorem \ref{bourabee} can be combined with the results of \cite{walks} to obtain the following result:
\begin{theorem}
\label{avbourabee}
Consider the set $B_N$ of all elements in the free group $F_k$ having length no more than $N$ in the generators. Then the \emph{average} index of the subgroup not containing a given element over $B_N$ is bounded above by a constant (which can be taken to be approximately $2.92$).
\end{theorem}

\begin{proof}[Proof Sketch]
Theorem \ref{bourabee} together with the results of \cite{walks} reduce the question to the same question, but with $F_k$ replaced by $\mathbb{Z}^k.$ Consider an element $x=(x_1, \dotsc, x_k) \in \mathbb{Z}^k.$ The element $x$ is \emph{not} contained in $p\mathbb{Z} \times \mathbb{Z}^{k-1}$ if $p$ does not divide $x_1.$ The result now follows by the Lemma \ref{primelem} below.
\end{proof}
\begin{lemma}
\label{primelem}
For every $n$ define $p(n)$ to be the smallest prime which does \emph{not} divide $n$. Then the expectation of $p(n)$ over all $n< N$ converges to $c=2.902...$ as $N$ tends to infinity.
\end{lemma}
\begin{proof}
For a fixed $,$ the probability that $p(n) = p,$ for some $n<N$ is given by 
\[
\dfrac{p-1}{p}/\prod_{q<p}q,
\]
where the product is over all the primes smaller than $p,$ and so the expectation of $p(n)$ is given by:
\[
\mathbb{E}(p) = \sum_{p}(1-p)/\prod_{q<p}q,
\]
and the latter sum converges rapidly to $2.92005...$ 
\end{proof}
\begin{remark}
A similar result was obtained independently and published (\cite{khalidben}) after this work appeared in preprint form 
\end{remark}
We note that the above argument \emph{does not} work if we replace the words \emph{index of subgroup} by \emph{index of normal subgroup}, since in that case we don't have the necessary control over the commutator subgroup. The bound given by Theorem \ref{expthm} is certainly not good enough -- we need a uniform estimate on the index of the normal subgroup of order  $o(n^k).$

We can get an estimate of the right type as follows:
\begin{theorem}
\label{nongamb}
Let $w$ be represented by a word of length $n$ in $F_2.$ There exists a prime $p\leq c n$ such that $w$ is not in the kernel of the homomorphism of $F_2$ to $\SL_2(\integers/p\integers)$ which maps the generators of $F_2$ to $g=\begin{pmatrix}1&2\\0&1\end{pmatrix}$ and $h=\begin{pmatrix}1&0\\2&1\end{pmatrix}$ respectively.
\end{theorem}
\begin{proof}
First, consider $g$ and $h$ to be elements of $\SL_2(\integers).$ The matrix given by the word $w(g, h)$ has entries no larger than $2^n$ in absolute value. By the Chinese Remainder Theorem, if $p_1, \dotsc, p_k$ are primes whose product exceeds $2^n,$ any such matrix which is the identity modulo all of the $p_i$ is, in fact, the identity matrix, and thus the word $w$ is the trivial element of $F_2$ (since $g$ and $h$ lie in the principal congruence subgroup of level $2$ in the modular group, and that principal congruence subgroup is free). By the prime number theorem, the product of all the primes not exceeding $m$ is asymptotic to $e^m$ for $m$ large, whence the result.
\end{proof}

\subsection{Residual finiteness and the Lower Central Series}
The method of proof of Theorem \ref{nongamb} is quite suggestive. For example, consider a word $w$ of length $n$ in $F_2=\langle a, b\rangle $ where the sum of the exponents of all the terms is \emph{not} zero. Then, by mapping both $a$ and $b$ to $g=\begin{pmatrix}1 & 1\\0 & 1\end{pmatrix}$ we see that the elements of $w(g, g)$ are no bigger than $n,$ and so such a word is not in the kernel of a homomorphism of $F_2$ into a group of order at most $\log n.$ If the exponent sum of $w$ is zero, but the individual exponent sums of $a$ and $b$ are not, we can modify the construction by sending $a$ to $g$ (as above) but sending $b$ to $g^2.$ This will give us the same result up to an additive constant. If the individual exponent sums are zero, the method fails, but the element $w$ is in the commutator subgroup (so $w$ is in the kernel of every homomorphism of $F_2$ to an abelian group). However, we can replace the abelian group by a $2$-step nilpotent group of ($3\times 3$) unipotent matrices. Using Lemma \ref{unipot}, we see  that  $w$ is in the second level of the lower central series of $F_2,$ it is not in the kernel of a homomorphism to a group of order $\log^3(n),$ where $n$ is the length of $w,$ and so on. We thus obtain the following statement which we believe sharp:
\begin{theorem}
\label{lowerng}
For every element $w$ of length $n$ at the $k$-th level in the lower central series of $F_2,$ there is a normal subgroup $H(w)$ of index $O(\log^{k^(k+1)/2}(n))$ which does not contain $w$ -- the implied constant in the big-$O$ notation depends on $k.$ The subgroup $H(w)$ is the kernel of a homomorphism onto a $k$-step nilpotent subgroup represented by $(k+1)\times (k+1)$ unipotent matrices.
\end{theorem}
\begin{remark}
Since $F_n$ is a finite index subgroup of $F_2$ (for any $n\geq 2$) we could have replaced $F_2$ by $F_n$ in the statement of Theorem \ref{lowerng}.
\end{remark}
We have used the following lemma:
\begin{lemma}
\label{unipot}
Let $m_1, \dotsc, m_k$ be $n\times n$ unipotent matrices. Let $M=w(m_1, \dotsc, m_k),$ where $w$ is a word of length $m.$ Then, the entries of $M$ grow no faster than $O(m^{n-1}),$ where the implicit constant in the $O$-notation depends only on the matrices $m_1, \dotsc, m_k.$
\end{lemma}
\begin{proof}
Write each $m_i$ as $m_i = I_n + m_i^\prime.$ The matrices $m_1^\prime \dotsc, m_k^{\prime}$ generate a nilpotent ideal $\mathcal{I},$ where $\mathcal{I}^n=0.$ This means that $w(m_1, \dotsc, m_k)$ can be written as sum of $j$-fold products of $m_i^\prime,$  where $j$ ranges between $0$ and $n-1.$ The number of $j$-fold products is at most $\binom{n}{j} = O(n^j),$ whence the result.
\end{proof}
The remaining question, then, is: how deeply in the lower central series can an element of word-length $n$ be? It is clear that a word of length $n$ cannot have depth greater than $n,$ so $k=O(n).$ It has been proved by J. Malestein and A. Putman in \cite{maleput} that $k=\Omega(\sqrt{n}).$
\subsection{Surface Groups}
\label{surfsec}
The first observation is that the proof of Theorem \ref{nongamb} does not depend on the fact that the group is free, but only on the existence of representations of the group into $\SL(2, \mathcal(O)),$ where $\mathcal(O)$ is the ring of integers in some algebraic number field -- instead of the prime number theorem we then use Landau's Prime Ideal Theorem (see, for example, \cite{montvaughan}) 
to get exactly the same estimate. To show that every surface group admits such a representation we use a (stronger) observation of C. Machlachlan and A. Reid (see \cite{MacReidCanJ}):
\begin{observation}
\label{picardobs}
The \emph{Picard Modular Group} -- $\PSL(2, \integers[i])$ -- contains the fundamental group of the compact surface of genus $2,$ and thus the fundamental group of every compact surface.
\end{observation}
This gives us the following version of Theorem \ref{nongamb}:
\begin{theorem}
\label{nongambsurf}
Let $w$ be represented by a word of length $n$ in  the fundamental group $\Gamma_g$ of the surface of genus $g.$ There exists a normal subgroup of order $O(n^3)$ which does not contain $w.$ \end{theorem}
The arithmetic method does not (at least not obviously) extend the proofs of Theorems \ref{lowerng} and \ref{expthm} to the case of surface groups. To extend these results we note that we need only extend these results to the fundamental group $\Gamma_2$ of the surface of genus $2.$ This is so because of the following observations:
\begin{observation}
\label{surfacefinite}
Every surface group $\Gamma_g$ s a subgroup of finite index of $\Gamma_2$ (in multiple ways, but we pick a fixed (for example, cyclic) covering for each $g$).
\end{observation}
\begin{observation}
\label{distobs}
By word-hyperbolicity, there exists a constant $c,$ such that every word of lengh $l$ in $\Gamma_g$ has word length between $c l$ and $l /c $ in $\Gamma_2.$ In fact, it is not hard to show that for the cyclic cover, the constant $c$ can be taken to be $g-1$ -- it can be shown that with a more judicious choice of covering this can be improved to $O(]\log g).$
\end{observation}
\begin{observation}
\label{surfacelcs}
If $g \in \Gamma_g$ lies at the $k$-th level of the lower central series of $\Gamma_g,$  then it lies in \emph{at most} the $k$-th level of the lower central series of $\Gamma_2.$ This is a more-or-less immediate corollary of the definition.
\end{observation}
\begin{observation}
\label{subgroupint}
\item Let $H$ be a subgroup of finite index $k$ in $\Gamma_2.$ Then $H \cap \Gamma_g$ is of index at most $k$ in $\Gamma_g.$ This is a standard exercise.
\end{observation}
To deal with $\Gamma_2,$ we use the following method, suggested by Henry Wilton: Write $\Gamma_2$ as $\Gamma_2 = \langle a, b, c, d ~\left| [a, b] = [c, d]\right.\rangle.$
There is a standard map of $r:\Gamma_2 \rightarrow F_2=\langle a, b\rangle,$ with $r(a)=r(c) = x,$ and $r(b)=r(d) = y.$ The retraction $r$ has the obvious property of not decreasing the lower central series depth, but it does have the unfortunate property of having a nontrivial kernel. However, this can be dealt with by using the following result, attributed by H. Wilton to G. Baumslag:
\begin{lemma}[\cite{wiltonex}[Lemma 4.13]]
\label{wiltonlem}
Let $\mathbb{F}$ be a free group, let $z \in \mathbb{F},$ $z\neq 1,$ and let 
\[
g = a_0 z^{i_1} a_1\dotsc,  z^{i_n}a_n,
\]
with $a_1, \dotsc, a_n \in \mathbb{F}.$ Assume further that $n\geq 1$ and $[a_k, z] \neq 1$ for $0< k<n.$ Then, if for every $1\leq k\leq n$ it is true that 
\[
|z^{i_k}|\geq |a_{k-1}| + |a_k| + |z|,
\]
then $g$ does not represent the trivial word in $\mathbb{F}.$
\end{lemma}

To use Lemma \ref{wiltonlem} we introduce the \emph{Dehn Twist automorphism} $\phi$ of $\Gamma_2,$ which is defined by $\phi(a) = a, \phi(b) = b, \phi(c) = c^{[a, b]}, \phi(d) = c^{[a, b]},$ and recall the following easy fact:
\begin{lemma}
\label{centralizer}
Consider an element $x$ in the fundamental group $G$ of a compact surface $S.$ Then $[x, y] = 1$ if and only if there exists an element $w$ and integers $k, l$ such that $x=w^k,$ $y=w^l.$
\end{lemma}
\begin{proof}
Represent $G$ as a Fuchsian group of isometries of $\mathbb{H}^2.$ Then, $x$ and $y$ are isometries of $\mathbb{H}^2.$ Since $S$ is a compact manifold, both $x$ and $y$ are hyperbolic elements, and since they commute, they have the same axis. Since the group $G$ is discrete, the translation distances of $x$ and $y$ are commensurate, and so $x=\gamma^m$ and $y=\gamma^n$ for some translation $\gamma$ (which is not necessarily in $G.$) However, by the obvious application of the Euclidean algorithm, $\beta=\gamma^{(m, n)} \in G,$ and $x=\beta^{m/(m, n)},$ and $y=\beta^{n/(m, n)},$ so $k=m/(m, n), l=n/(m, n)$ and $w=\beta$ as stated.
\end{proof}
\begin{remark} A completely different proof of Lemma \ref{centralizer}, based on a deep topological result of J.~H.~C.~Whitehead (\cite{whitehead}), is given by W. Jaco in \cite{jacofree}. The result of Whitehead is (in essence) that every infinite index subgroup of a closed surface group is a free group. The result of Lemma \ref{centralizer} then follows, by observing that $<x, y>$ is both free and abelian, hence cyclic.
\end{remark}

Now we are ready to prove the key observation
\begin{theorem}
\label{limitgp}
Let $g\in \Gamma_2,$ $g\neq 1.$ Then $r(\phi^{l(g)/4}(g) )\neq 1,$ where $l(g)$ is the minimal length of $g$ in terms of the generating set $\{a, b, c, d\}.$
\end{theorem}
\begin{proof} Let $w(g)$ be a shortest word in $a, b, c, d$ representing $g.$
First, write $w(g) = L_1 R_1 \dots L_k R_k,$ where $L_j$ are blocks of $a$s and $b$s and $R_i$ are blocks of $b$s and $c$s. Further, let $z_1 = [a, b]$ and $z_2 = [c, d]$ (these represent the same element $z$ of $\Gamma_2,$ but we think of them as words for now.) Now, apply the following rewriting process (see Algorithm \ref{gen1}) to $w(g)$: 

First, if any $L_i$ is a power of $z_1$, replace it by the same power of $z_2.$ Second, if any $R_j$ is a power of $z_2,$ replace it by the same power of $z_1.$ Then repeat the two steps until neither can be applied. Note that this process will terminate eventually, since each steps reduces the total number of blocks (in fact, it reduces \emph{both} the number of $L$ blocks and the number of $R$ blocks). Call the resulting word $w_0(g)$ ($|w_0(g)| = |w(g)|,$ since $w(g)$ was assumed minimal). By abuse of notation, let 
$w_0(g) = L_1 R_1\dots L_k R_k$ (where the $k$ might be different from the $k$ in $w(g)$). Note that applying the automorphism $\phi^m$ to $g$ replaces 
each occurrence of a block $R_j$ in $w_0$ by $z^{-m} R_j z^m,$ and hence applying
 $r\circ \phi^m$ to $g$ maps $g$ to 
the word $u_0(g)=L_1 z_1^{-m} r(R_1)^m L_2 z_1^{-m} r(R_2) z_1^m \dots L_k z_1^{-m} r(R_k) z_1^m.$ By construction of $w_0$ and Lemma \ref{centralizer}, the hypotheses of Lemma \ref{wiltonlem} hold, as long as $4(m-1)\geq |w_0(g)|.$
\end{proof}
\begin{algorithm}
\label{gen1}
\caption{Rewriting Algorithm}
\begin{algorithmic}[1]
\LOOP
\IF {$w(g) = z_1^p$}
\STATE Return $w(g).$
\ELSIF{Any $L_i$ is a power of $z_1,$ so that $L_i=z_1^{m_i}$}
\STATE  Replace $L_i$ by $z_2^{m_i}.$ 
\ELSIF {any $R_j$ is a power of $z_2,$ so that $R_j=z_2^{n_j}$}
\STATE Replace $R_j$ by $z_1^{n_j}.$
\ENDIF
\ENDLOOP
\end{algorithmic}
\end{algorithm}

\begin{lemma}
\label{wordlen}
The length of $r(\phi^{l(g)/4}(g)$ is bounded above by $l(g)^2+l(g).$
\end{lemma}
\begin{proof}
Computation.
\end{proof}
As a corollary of Theorem \ref{limitgp} and Lemma \ref{wordlen}, we have the following extensions of Theorems \ref{lowerng} and \ref{expthm} respectively:
\begin{theorem}
\label{lowerngsurf}
For every element $w$ of length $n$ at the $k$-th level in the lower central series of  the fundamental group of a closed surface of  genus $g$ there is a normal subgroup $H(w)$ of index $O(\log^{k(k+1)/2}(n))$ which does not contain $w$. 
\end{theorem}

\begin{theorem}
\label{expthmsurf}
Consider the fundamental group $\Gamma_k$ of a surface of genus$g$, and let $n  > 1.$  There exists a normal subgroup $H_n$ of $\Gamma_k$ of index $f(n)$ which contains \emph{no} non-trivial elements of word length smaller than $n,$  where the index $f(n)$ can be bounded by
\[
f(n) =O(g^{O(n^2)}).
\]
\end{theorem}
I do not expect that the bound in the statement of Theorem \ref{expthmsurf} is close to sharp (but it \emph{is} a bound).

\section{Regular coverings}
\label{regcoversec}
D. Futer asked whether the results of the previous section had analogues when the covering given by $\Pi$ was additionally required to be \emph{regular}. This seems to be a hard question in general.In the case where $S$ is a planar surface with $k$ boundary components, we have the following result:
\begin{theorem}
\label{regplanar}
In order for a covering of the boundary of $S$ to extend to a \emph{regular} covering of $S,$ it is necessary and sufficient that, in addition to the requirements of Theorem \ref{sullthm} (part 1), there must be a subgroup $G < S_n,$ with $|G| = n$ and  $G$ is generated by $\gamma_1, \dots, \gamma_k,$ where $\gamma_i \in \Gamma_i,$ for $i=1, \dots, k.$
\end{theorem}

As far as the author knows, there is no particularly efficient way of deciding whether the condition of Theorem \ref{regplanar} is satisfied.

Here is a more satisfactory (in not very positive) result:
\begin{theorem}
\label{regq}
Let $S$ be a surface with one boundary component. There does not exist a nontrivial covering of finite index  $\Pi: \widetilde{S}\rightarrow S$ where $\widetilde{S}$ also has one boundary component.
\end{theorem}

\begin{proof}
Let the degree of the covering be $n.$ If $\widetilde{S}$ has one boundary component, the generator of the $\pi_1(\partial S)$ gives rise to the cyclic group $\mathbb{Z}/n\mathbb{Z},$ which is a subgroup of the deck group of $\Pi.$. Since the deck group has order $n$ (by regularity), the covering is cyclic (so that the deck group is, in fact, $\mathbb{Z}/n \mathbb{Z}.$). A cyclic group is abelian, and since the generator of $\pi_1(\partial S)$ is a product of commutators, it is killed by the map $\Sigma.$ But this contradicts the statement of the first sentence of this proof (that this same element generates the entire deck group).
\end{proof}

The proof also shows the following:
\begin{theorem}
\label{regq2}
Let $\Pi: \widetilde{S}\rightarrow S,$ where $S$ has a single boundary component, be an \emph{abelian} regular covering of degree $n.$ Then $\widetilde{S}$ has $n$ boundary components.
\end{theorem}

We can combine our results in the following omnibus theorem:
\begin{theorem}
\label{regq3}
Let $S$ be a surface, whose boundary has $k$ connected components. Let the conjugacy classes of the $k$ coverings be be $\Gamma_1, \dotsc, \Gamma_k.$  In order for a constant cardinality $n$ covering of $\partial S$ to extend to a regular covering of $S$ it is necessary and sufficient that there be elements $\gamma_1\in \Gamma_1, \dotsc, \gamma_k \in \Gamma_k$ such that $\gamma_1, \dotsc, \gamma_k$ generate an order $n$ subgroup of $S_n$ and $\gamma_1\dots \gamma_k = e.$
\end{theorem}

%
%
%

\bibliographystyle{plain}
\bibliography{oneint,rivin}
\end{document}